\documentclass[11pt]{article}
\usepackage{amsfonts,amsmath,amssymb,amsthm, amscd, subcaption}
\usepackage{enumerate}
\usepackage{graphicx}
\usepackage[stable]{footmisc}
\numberwithin{equation}{section}

\newtheorem{theorem}{Theorem}[section]

\newtheorem{cor}[theorem]{Corollary}

\theoremstyle{definition}

\newtheorem{example}[theorem]{Example}

\def\<{{\langle}}
\def\>{{\rangle}}
\def\G{{\Gamma}}
\def\a{{\alpha}}
\def\b{{\beta}}

\def\d{{\delta}}
\def\Z{\mathbb Z}

\def\a{\alpha}
\def\b{\beta}

\def\G{\Gamma}

\def\s{\sigma}
\def\e{\epsilon}

\def\De{{\Delta}}

\def\ni{\noindent} 

\begin{document}

\title{Knot Invariants from Laplacian Matrices}

\author{Daniel S. Silver
\and Susan G. Williams\thanks {The authors are partially supported by the Simons Foundation.} }

\maketitle 


\begin{abstract} 

A checkerboard graph of a special diagram of an oriented link is  made a directed, edge-weighted graph in a natural way so that a principal minor of its Laplacian matrix is a Seifert matrix of the link. Doubling and weighting the edges of the graph produces a second Laplacian matrix such that a principal minor is an Alexander matrix of the link. The Goeritz matrix and signature invariants are obtained in a similar way.  A device introduced by L. Kauffman makes it possible to apply the method to general diagrams. \end{abstract} 
 \bigskip

\begin{center} MSC 2010: 57M25, 57M15 \end{center}

\section{Seifert matrices} \label{Intro} Classical link invariants are often computed from a Seifert matrix associated to a \textit{Seifert surface}, an orientable surface $\Sigma$ with boundary equal to the link $\ell$. One begins with oriented curves $a_1, \ldots, a_n$ in $\Sigma$ that represent a  basis for the homology group $H_1(\Sigma; \Z)$. 
A \textit{Seifert matrix} $V = (V_{i,j})$ is defined as the $n \times n$ matrix 
with $V_{i,j}$ equal to the linking number ${\rm Lk}(a_i, a_j^+)$, where $a_j^+$ is a copy of $a_j$ pushed off of $\Sigma$ in the direction of a positive normal vector. 

One of the invariants of $\ell$ that can be obtained from $V$ is the \textit{Goeritz matrix} $G= V+V^T$ with $V^T$ the transpose of $V$, provided that $V$ arises from a ``special Seifert surface" (see below). It is well defined up to transformations $G \mapsto RGR^T$ with unimodular $R$, and elementary enlargements (or their inverses) $G \mapsto {\rm diag}(G, \pm 1)$. Another is the single-variable \textit{Alexander polynomial} $\Delta_\ell(x)= {\rm Det}(V - x V^T)$, well defined up to multiplication by $\pm x^i, i \in \Z$. Finally, if $\omega$ is any unit-modulus complex number, $\omega \ne 1$, then the signature of $(1-\omega)V+(1-\overline \omega) V^T$ is the \textit{$\omega$-signature $\sigma_\omega(\ell)$}. Details about each of these invariants can be found in \cite{Li97}. \medskip

The purpose of this note is to show how Seifert matrices and and invariants derived from them arise naturally from Laplacian matrices of directed graphs associated to link diagrams. The ideas here motivated \cite{STW18} 
in which the Seifert matrix and Goeritz matrix of a link are recovered from 
a modified Dehn presentation of the link group.  \medskip

The \textit{Laplacian matrix} $L(\G)$ of a directed graph $\G$ is the square matrix indexed by the vertex set of $\G$ with diagonal entries $L(\G)_{i,i}$ equal to the \textit{weighted out-degree} of the vertex $v_i$, the sum of weights of edges with initial vertex $v_i$, and  non-diagonal entries $L(\G)_{i,j}$ equal to $-1$ times the sum of the weights of edges from $v_i$ to $v_j$.

A \textit{diagram} $D$ of a link $\ell$ is a generic projection of the link in the plane, a regular 4-valent graph, with extra information at each vertex indicating which arc of the link passes over the other. We assume throughout that $D$ has no nugatory crossings, and it cannot be separated by any circle in the plane. A standard operation of ``smoothing"  crossings of $D$, as in Figure \ref{smoothing}, results in a collection of Seifert disks, possibly nested, in the plane that can be rejoined by half-twisted bands of $D$ to produce a Seifert surface $\Sigma$ for $\ell$. The diagram $D$ is \textit{special} if the disks are not nested.

\begin{figure}
\begin{center}
\includegraphics[height=1 in]{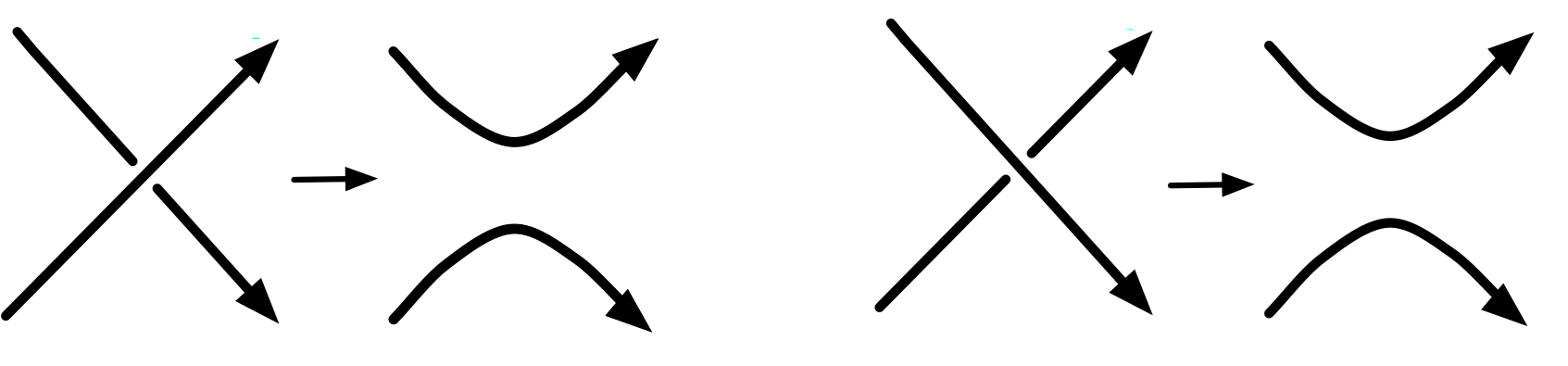}
\caption{Smoothing a crossing of a diagram $D$.}
\label{smoothing}
\end{center}
\end{figure}

Any diagram $D$ can be deformed to a special diagram by lifting nested Seifert disks and pulling them across bands (see \cite{BZ03}, page 229). While conceptually simple, the process is often messy. Alternatively, a method of L. Kauffman (\cite{Ka87}, page 185) enables us to add unknotted, unlinked components to $D$ and thereby create a special diagram. The method is reviewed in Section \ref{tracer}. 

By a \textit{region} of a link diagram $D$ we mean a face of the underlying projection of the link. The diagram can be checkerboard shaded so that every edge separates a shaded region from an unshaded one. For the sake of definiteness we adopt the convention that the unbounded region is unshaded. The diagram is special if and only if its shaded regions form a Seifert surface of the link.

It is common to consider the \textit{shaded checkerboard graph} of a plane graph $D$. It has vertices corresponding to the shaded regions of $D$, and an edge between two vertices for every crossing at which the corresponding regions meet. The diagram is special if and only if the graph is bipartite. 

Here we consider the dual graph $\G= \G(D)$, the \textit{unshaded checkerboard graph} of $D$, with vertices corresponding to the unshaded regions of $D$. The vertex corresponding to the unbounded region will be denoted by $v_\infty$. As before, there is an edge joining two vertices for every crossing $c$ where the corresponding regions meet. Direct this edge and label it with the \textit{weight} $w_e$ equal to $\theta(c)=\pm 1$, as in Figure \ref{directed}. 

\begin{figure}
\begin{center}
\includegraphics[height=1.6 in]{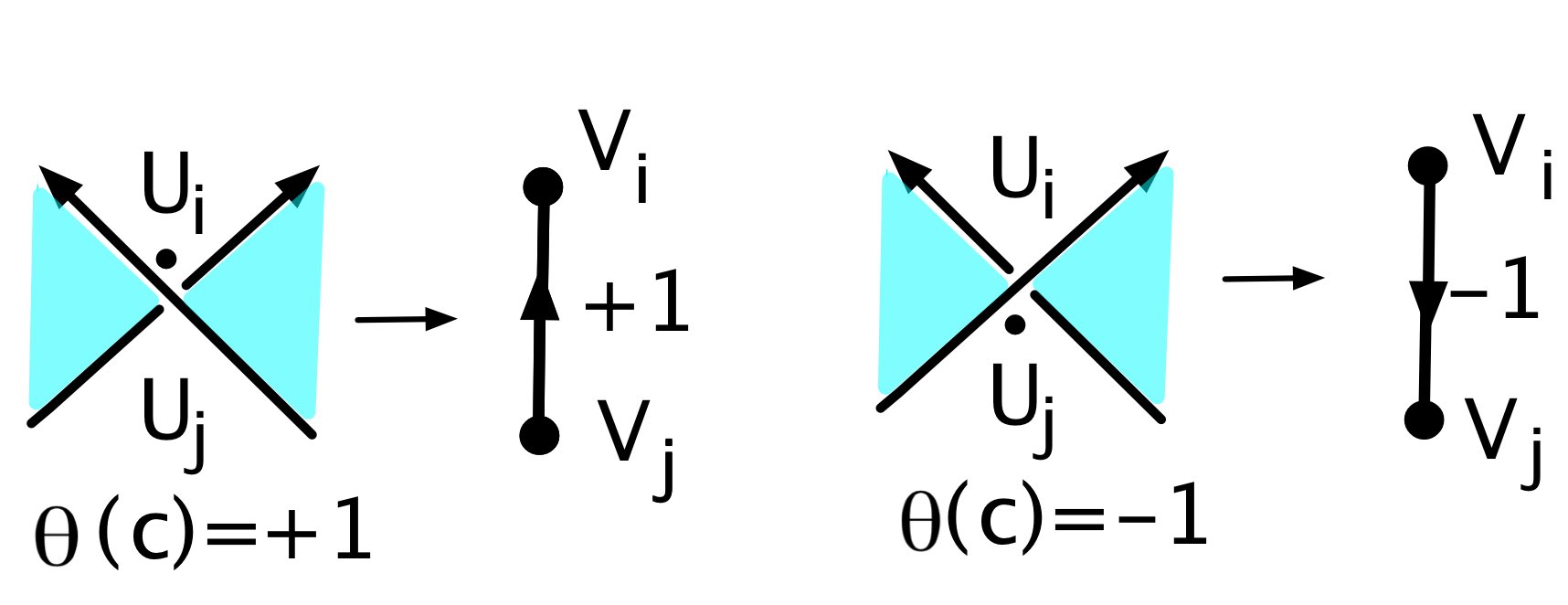}
\caption{Constructing the graph $\G(D)$ from a special diagram $D$.}
\label{directed}
\end{center}
\end{figure}

Define $\G^{\a, \b}$ to be the graph obtained from $\G$ by doubling and weighting its edges in the following way. For every edge $e$ of $\G$ from $v_i$ to $v_j$, replace the weight $w_e$ by $w_e \a$; add a new edge $\bar e$ from $v_j$ to $v_i$ with weight $w_{\bar e} = w_e \b$.

We denote by $L_{v_\infty}(\G)$ (resp. $L_{v_\infty}(\G^{\a, \b})$)  the principal minor of $L(\G)$ (resp. $L(\G^{\a, \b})$) gotten by deleting
the row and column corresponding to the vertex $v_\infty$.

\begin{theorem}\label{main} Assume that $D$ is a special diagram of an oriented link $\ell$ with unshaded checkerboard graph $\G$. Then 
\begin{enumerate}[(i)]
\item $L_{v_\infty}(\G)$ is a Seifert matrix $V$ for $\ell$; 
\item $L_{v_\infty} (\G^{1, 1})$ is a Goeritz matrix $G$ for $\ell$. 
\item $L_{v_\infty} (\G^{1, -x})$ is an Alexander matrix $A$ for $\ell$.  
\end{enumerate}
\end{theorem} 

From the first part of Theorem \ref{main} we obtain:

\begin{cor} \label{signature} If $\G$ is the unshaded checkerboard graph of a special diagram for $\ell$, and $\omega$ is a unit-modulus complex number, $\omega \ne 1$, then the signature of the Hermitian matrix $L_{v_\infty} (\G^{1-\omega, 1- \overline \omega})$ is the $\omega$-signature $\sigma_\omega(\ell)$. \end{cor}

The following corollary was proven for alternating special diagrams in \cite{MS03}. The general result below follows from Theorem \ref{main} and the Matrix Tree Theorem for directed graphs using the observation that spanning trees of $\G$ correspond to spanning trees of $\G^{1, -x}$ that are directed toward $v_\infty$. (We are told by H. Russell \cite{Ru18} that the general result can also be obtained using results of \cite{CDR14}). 
\begin{cor} \label{treesum} If $\G$ is the  unshaded checkerboard graph of a special diagram for $\ell$, then the Alexander polynomial of $\ell$ is given by
\begin{equation} \De_\ell (x) = \sum_T \prod_{e \in E_T} \widetilde w_{e,T}, \end{equation}
where $T$ ranges over all spanning trees of $\G$, $E_T$ denotes the edge set of $T$, and 
\begin{equation} \widetilde w_{e,T} = \begin{cases}w_e & \text{if $e$ is directed towards $v_\infty$ in the tree $T$}\\
									-xw_e  & \text{otherwise.}
\end{cases}\end{equation}

\end{cor}
 
 \begin{example} Figure \ref{toruslink} displays a checkerboard graph $\G$ associated to the $(2,6)$-torus link suitably oriented. Each of its 6 spanning trees is obtained by deleting a single edge of the graph. Removing edge $e_i$ results in a spanning tree $T$ with exactly $i-1$ edges directed away from $v_\infty$. By Corollary \ref{treesum} the Alexander polynomial of the link is $(-x)^0 + (-x)^1 +  (-x)^2 + (-x)^3+ (-x)^4+ (-x)^5 = 1 - x + x^2 -x^3 + x^4 - x^5$.
The example immediately generalizes for any $(2, n)$-torus link.
\end{example} 

In order to prove Theorem \ref{main} we describe a combinatorial method in \cite{BZ03} (see page 231), motivated by \cite{Mu65}, for computing the Seifert matrix $V=(V_{i,j})$ associated to the special diagram $D$ of the oriented link $\ell$. 

The shaded regions of $D$ form a Seifert surface $\Sigma$ for the link. The boundaries $\partial U_i$ of the unshaded regions, oriented in the counterclockwise sense, are a basis for $H_1(\Sigma; \Z)$ (see, for example, \cite{BZ03}). 
We can compute the linking numbers $V_{i, j} = {\rm Lk}(\partial U_i, \partial U_j^+)$ by a combinatorial procedure: 

Following each link component in the preferred direction, at each crossing $c$ we place a dot in the corner of the unshaded region that is to the left of $c$, as in Figure \ref{directed}. For each $i=1, \ldots, n$ and each crossing $c$, define: 

\begin{equation} \e_i(c) = \begin{cases}1 & \text{if $U_i$ contains a dot at $c$}\\
									0  & \text{otherwise.} 
									
									\end{cases} \end{equation}
Then \begin{equation}\label{formula}
V_{i,i} = \sum_{c \in \partial U_i} \theta(c) \e_i(c), \quad V_{i,j} = -\sum_{c \in \partial U_i \cap \partial U_j} \theta(c) \e_j(c). \end{equation} (The reader is forewarned that the minus sign in this last equation is missing in \cite{BZ03}, likely due to a typographical error.) \medskip
  \begin{figure}
\begin{center}
\includegraphics[height=2 in]{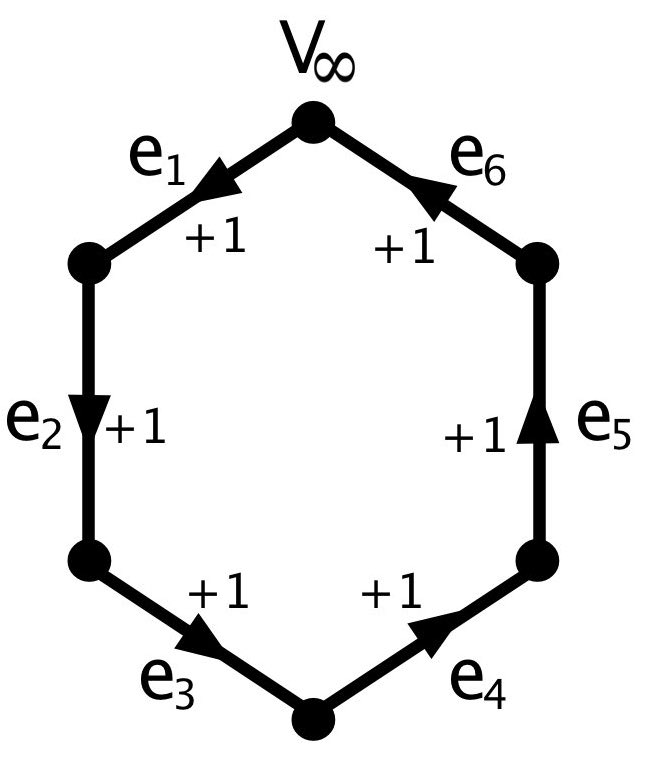}
\caption{Checkerboard graph $\G(D)$ of $(2, 6)$-torus link}
\label{toruslink}
\end{center}
\end{figure}

Consider an unshaded checkerboard graph $\G$ with edges weighted and directed as in Figure \ref{directed}.  We see easily that for $i\ne j$ the crossing on the left contributes $-1$  to both $V_{j,i}$ and $L(\G)_{j,i}$, while the crossing on the right contributes  $+1$ to both $V_{i,j}$ and $L(\G)_{i,j}$.  

We also see that the diagonal entry $V_{i,i}$ is the weighted in-degree of the vertex $v_i$ of $\G$.  Recall that $L(\G)_{i,i}$ is the weighted out-degree.  To prove statement (i) of the theorem, it remains to see that these are equal.  In the case that all indices $\theta(c)$ are $+1$,  this holds because $\G$ is a directed Eulerian graph, with in- and out-edges alternating about every vertex.  If we change the weight of an edge to $-1$ we also reverse its direction.  One easily checks that this has the same effect on the in-degree and out-degree of each of the incident vertices.

The remaining statements follow from the fact that reversing all edge directions of $\G$ has the effect of transposing $L(\G)$. This is clear for non-diagonal entries; for diagonal terms we use again the observation that weighted in-degree and out-degree are equal. Consequently, $L_{v_\infty}(\G^{\a, \b})=\a V+ \b V^T$, for any $\a, \b$. 
In particular, $L_{v_\infty}(\G^{1, 1})= V+V^T$ is a Goeritz matrix for $\ell$. Similarly, $L_{v_\infty}(\G^{1, -x})= V-x V^T$ is an Alexander matrix for $\ell$.

\begin{example} Consider the special diagram of the knot $k = 5_2$ in Figure \ref{ex1}. The index $\theta(c)$ is $-1$ for every crossing $c$. The associated directed graph $\G$ appears on the right with ordered vertex set $\{v_1, v_2, v_\infty\}$. Its Laplacian matrix is 
$$L(\G) = \begin{pmatrix} -1 & 0 & 1 \\ 1 & -2 & 1 \\ 0 & 2 & -2 \end{pmatrix}.$$
The reader can verify that the principal $2 \times 2$ submatrix ]
$$L_{v_\infty}(\G)=\begin{pmatrix} -1 & 0 \\ 1 & -2  \end{pmatrix}$$
is a Seifert matrix $V$ for the knot, while
$$L_{v_\infty}(\G^{1, -x})= \begin{pmatrix} -1+x & -x \\ 1 & -2+2x  \end{pmatrix}$$
is an Alexander matrix. 

\begin{figure}
\begin{center}
\includegraphics[height=2 in]{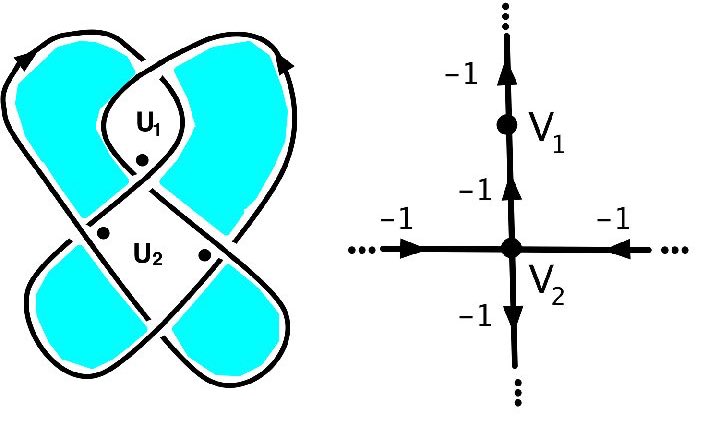} 

\caption{Special diagram for $k= 5_2$ (left) and checkerboard graph $\G$. The vertex $v_\infty$ is not shown.}
\label{ex1}
\end{center}
\end{figure}

\end{example}

\section{Kauffman's tracer circuits} \label{tracer}  We can describe an arbitrary diagram $D$ of an oriented link $\ell$ schematically with oriented circles that represent the Seifert circles of $D$; they are joined by short arcs representing the half-twisted bands that join the circles. The arcs are labeled $+1$ or $-1$ according to Figure \ref{newdirected}. An example appears in Figure \ref{ex2}. We will call such a diagram a \textit{Seifert circle diagram} of $\ell$. 

\begin{figure}
\begin{center}
\includegraphics[height =1in]{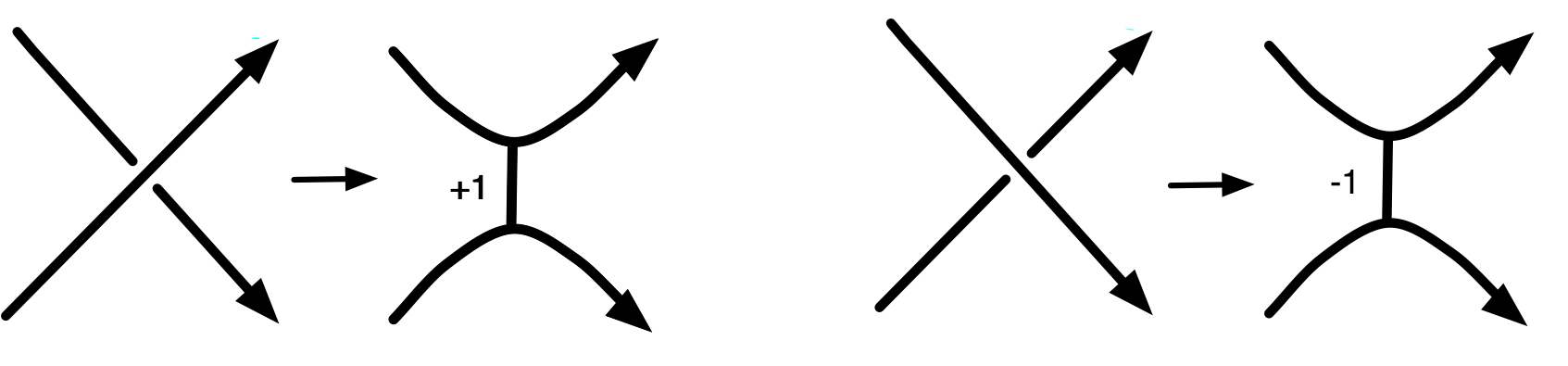}
\caption{Creating Seifert surface diagram.}
\label{newdirected}
\end{center}
\end{figure}

\begin{figure}
\begin{center}
\includegraphics[height =1.5 in]{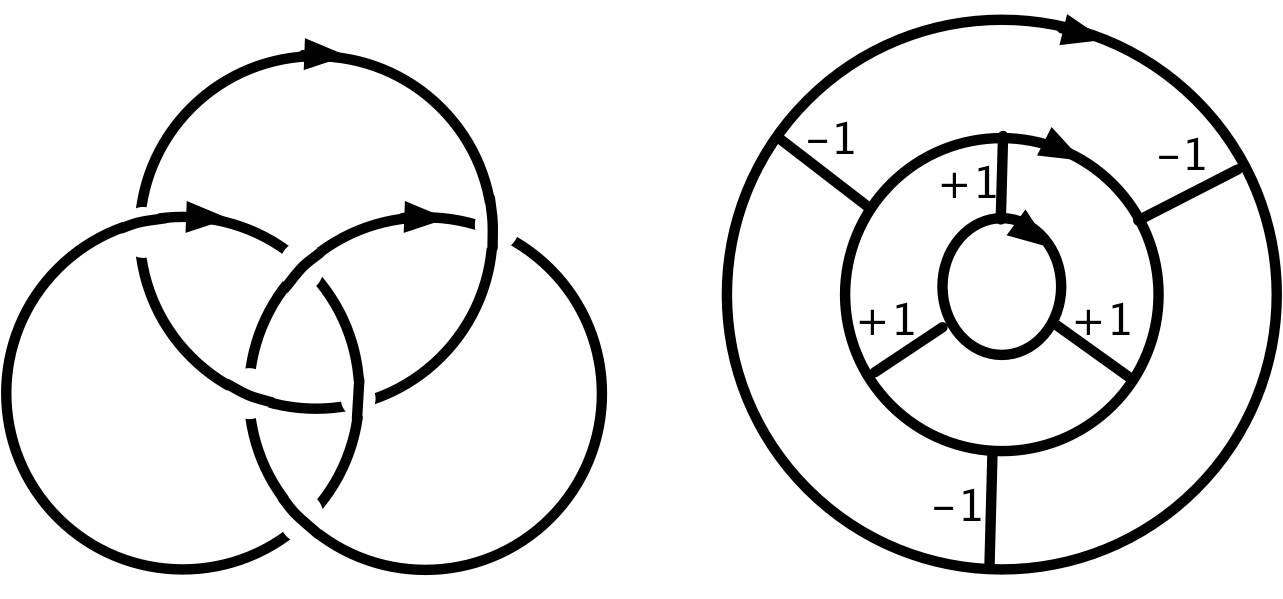}
\caption{Non-special diagram of Borromean rings (left) and Seifert circle diagram (right).}
\label{ex2}
\end{center}
\end{figure}

The Seifert circles $\s$ bound disks $\d$, not necessarily disjoint, in the plane. If the diagram $D$ is not special then at least one disk contains another in its interior. Consider such a disk $\d$, with $\partial \d = \s$. 

We add an unknotted component to $\ell$, drawn in $D$ as an unknotted circle that follows a parallel copy of $\s$, staying in $\d$ except where it detours around a crossing to an interior Seifert circle $\s'$,  always underneath the other arcs of $D$. Orient this circle in the opposite direction of $\s$.  The new Seifert circle diagram can be gotten from the original by the following procedure. 
Let $p_1, \ldots, p_r$ be the intersection points of $\s$ and arcs joining $\s$ with Seifert circles in its interior. Add points $q_1, \ldots, q_r$ along $\s$ interspersed with $p_1, \ldots, p_r$. ``Blow up" each point $p_i, q_i$ to a small circle; that is, replace a small neighborhood  in $\s$ of each point $p_i, q_i$ with a circle $\s_i, \s'_i$ respectively. The circles are joined to each other by the remaining segments of $\s$. Give $\s_i$ (resp. $\s_i'$) the opposite (resp. same) orientation of $\s$. The arc to the left of each $\s_i$ receives weight $+1$, while the arc to the right receives $-1$.  The arc that originally connected $p_i$ to an interior circle now attaches to $\s_i$.  Any arc that originally connected a segment of $\s$ between $p_i$ and $p_{i+1}$ to a Seifert circle exterior to $\s$ now attaches to the circle $\s_i'$. The procedure is illustrated in Figure \ref{blowup}. (The sign of the crossing of the arcs of $D$ can be reversed without affecting the right-hand side of the figure.) 

\begin{figure}
\begin{center}
\includegraphics[height =1.5 in]{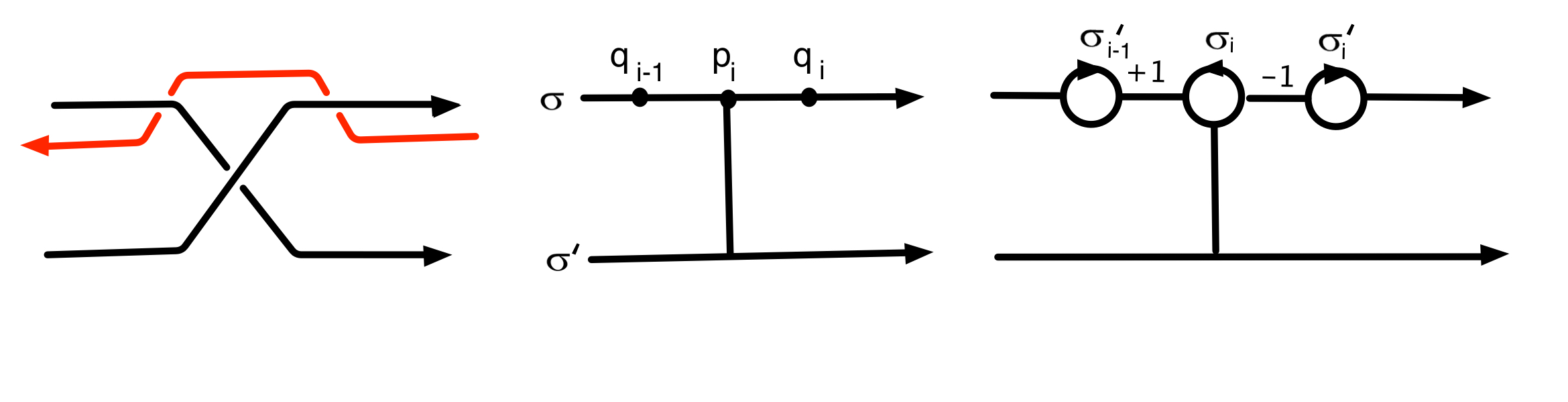}
\caption{Detail of link diagram and tracer circle (left), detail of Seifert circle diagram (center), points $q_{i-1}, p_i, q_i$ blown up (right).}
\label{blowup}
\end{center}
\end{figure}

The procedure can be repeated as often as necessary so that the final diagram $D^\sharp$ is special. We can obtain the unshaded checkerboard graph of $D^\sharp$ from its  Seifert circle diagram  by placing a vertex in each non-circle region and joining pairs of vertices by an edge whenever the corresponding regions are separated by an arc. If the separating arc is weighted $+1$ (resp. $-1$) then the edge is weighted $-1$ (resp. $+1$). The direction of the edge follows the direction of the two circles on either side if its weight is $+1$; otherwise it follows the opposite direction. 

\begin{example}\label{eight} Viewing Figure \ref{eightfig} from the top, we see a diagram $D$ of the oriented figure-eight knot $k=4_1$, its Seifert circle diagram, and the new Seifert circle diagram of $D^\sharp$. The unshaded checkerboard graph $\G(D^\sharp)$, obtained here directly from $D^\sharp$, is displayed 
in the second row of the diagram. The edges leading away from $v_2$ and $v_4$ go to the vertex $v_\infty$. 

The Laplacian matrix of $\G(D^\sharp)$ is
$$\begin{pmatrix} -1 & 1 & 1 &-1 & 0\\ 
0 & 1 & 0 & -1 & 0 \\ 
1&-1 & -1 & 1 & 0\\
0 & -1 & 0 & 1 & 0 \\
0 & 0 & 0 & 0 & 0 \end{pmatrix}.$$
The principal $2 \times 2$ submatrix 
$$V=\begin{pmatrix}  
-1 & 1  \\ 
0&1  \end{pmatrix}$$
obtained by deleting rows and columns corresponding to $v_3, v_4$ and $v_\infty$ is a Seifert matrix for the link formed by $k$ and two unknotted, unlinked components. The Alexander matrix $V - x V^T$ for the link arises in a similar way from the doubled graph $\widetilde \G(D^\sharp)$. This is an instance of a general result, Theorem \ref{tracer}. 
\end{example} 

\begin{figure}
\begin{center}
\includegraphics[height =3in]{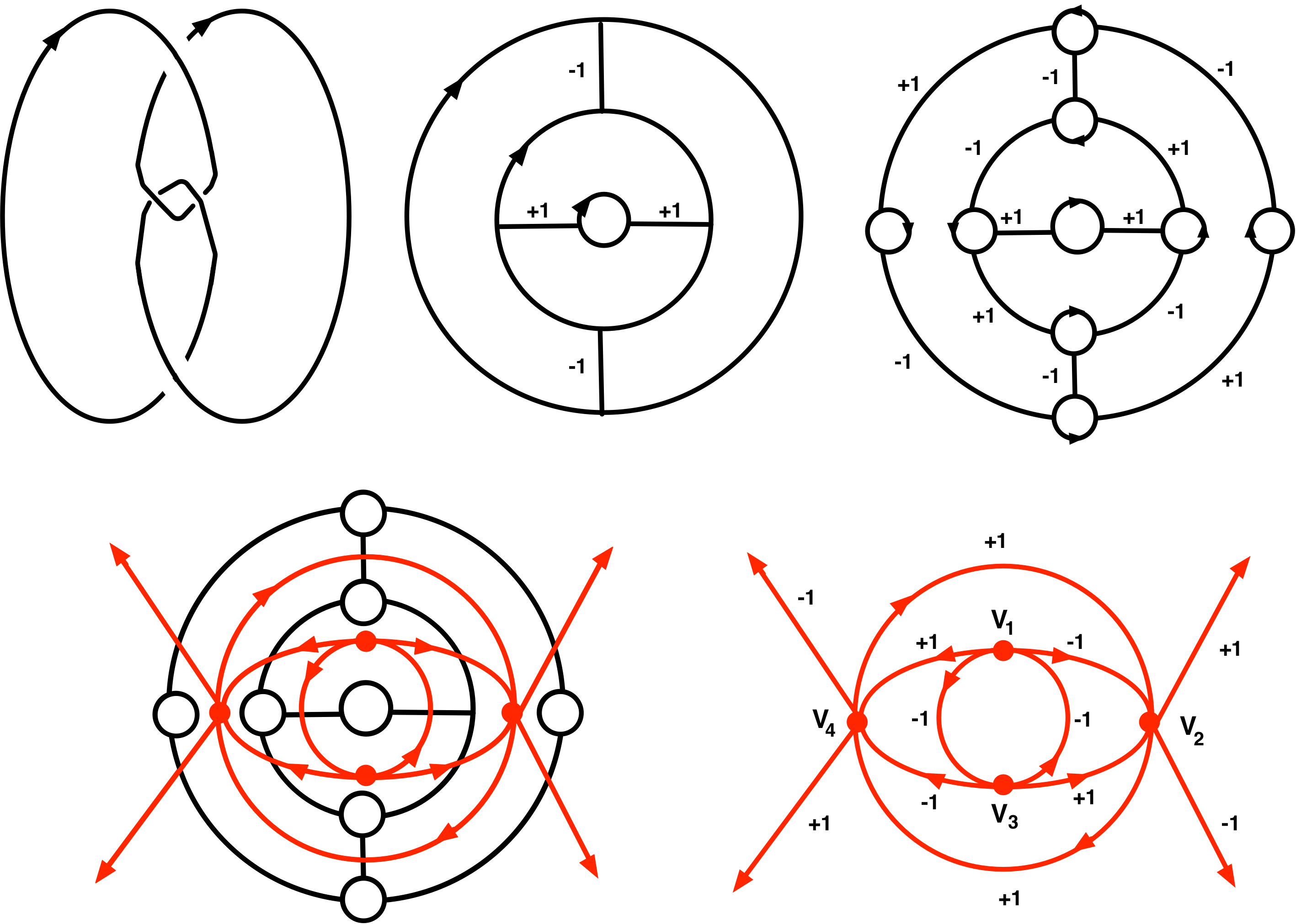}
\caption{Diagram $D$ of oriented figure-eight knot (top, left); associated Seifert circle diagram (top, center); Seifert circle diagram for $D^\sharp$ (top, right); unshaded checkerboard graph $\G(D^\sharp)$ (red, bottom, left); $\G(D^\sharp)$ with edge weights and labeled vertices.}
\label{eightfig}
\end{center}
\end{figure}

For $W$ a subset of the vertex set of $\G(D^\sharp)$, we denote by 
$L_W(\G(D^\sharp))$ the principle submatrix of the Laplacian matrix $L(\G(D^\sharp))$ obtained by deleting the rows and columns corresponding to elements of $W$. We define $L_W(\G^{\a, \b}(D^\sharp))$ similarly. 
A spanning forest $F$ of $\G(D^\sharp)$ is \textit{rooted at $W$} if each connected component of $F$ contains a unique element of $W$. 

\begin{theorem} \label{tracer} Assume that $D$ is a diagram of an oriented link $\ell$ and $D^\sharp$ is a special diagram obtained by adding $\tau$ tracer circles. Let $W = \{v_{i_1}, \ldots, v_{i_\tau}, v_\infty\}$, where, for each $j = 1, \ldots, \tau$, $v_{i_j}$ is a vertex of $\G(D^\sharp)$ corresponding to a region of $D^\sharp$ that meets the $j$th tracer circle. Then
the conclusions of Theorem \ref{main} and Corollary \ref{signature} hold with $L_{v_\infty}$ replaced with $L_W$. 
%
%
\end{theorem}

\begin{proof} The Seifert surface $\Sigma^\sharp$ corresponding to $D^\sharp$ can be obtained from $\Sigma$ by puncturing each nesting Seifert disk. The homology group $H_1(\Sigma^\sharp; \Z) \cong H_1(\Sigma; \Z) \oplus \Z^\tau$ is is free with basis $\cal B$ represented by the boundaries of the unshaded regions of $D^\sharp$. 

Consider an outermost nesting disk $\d$ and associated tracer circle $\s$ in $D^\sharp$. Let $U_{\d, 1}, \ldots, U_{\d, s}$ be the bounded unshaded regions that meet $\s$. The sum $\partial U_{\d, 1} + \cdots + \partial U_{\d, s}$ is null-homologous in $\Sigma$. We recover a basis for $H_1(\Sigma; \Z)$ from $\cal B$ by discarding any one of $\partial U_{\d, 1}, \ldots, \partial U_{\d, s}$, and repeating for each tracer circle. The desired results follow as in the proofs of Theorem \ref{main} (ii) and Corollary \ref{signature}.
\end{proof} 

Applying the All Minors Matrix Tree Theorem \cite{Ch82} to Theorem \ref{tracer} we obtain the following.

\begin {cor}
 The Alexander polynomial of $\ell$ is given by
\begin{equation} \De_\ell (x) = \sum_F \prod_{e \in E_F} \widetilde w_{e,F} \end{equation}
where $F$ ranges over all spanning forests of $\G^\sharp$ rooted at $W$, and 
\begin{equation} \widetilde w_{e,F} = \begin{cases}w_e & \text{if $e$ is directed towards $W$}\\
									-xw_e  & \text{otherwise.}
\end{cases}\end{equation}
\end{cor}

%
%
%

\bigskip

\ni Department of Mathematics and Statistics,\\
\ni University of South Alabama\\ Mobile, AL 36688 USA\\
\ni Email: silver@southalabama.edu, swilliam@southalabama.edu

\begin{thebibliography}{1}

\bibitem{Bo60} C.~W. Borchtard, Ueber eine der Interpolation entsprechende Darstellung der Eliminations-Resultante, {\it J. Reine Angew. Math. \bf 57} (1860),  111--121.

\bibitem{BZ03} G. Burde and H. Zieschang, {\it Knots,} 2nd ed., Walter de Gruyter, Berlin, 2003.

\bibitem{Ch82} S. Chaiken, A Combinatorial proof of the all minors matrix tree theorem, 
{\it SIAM  J. Alg. Disc. Meth. \bf 3} (1982), 319--329.

\bibitem{CDR14} M. Cohen, O. Dasbach and H.~M. Russell, A twisted dimer model for knots, {\it Fund. Math. \bf 225} (2014), 57--74. 

\bibitem{Ka87} L.~H. Kauffman, {\it On knots, } Annals of Mathematics Studies, Princeton University Press, Princeton, NJ, 1987.

\bibitem{Li97} W.~B. Lickorish, An Introduction to Knot Theory, Springer-Verlag, New York, 1997.

\bibitem{Mu65} K. Murasugi, On a certain numerical invariant of link types, {\it Trans. Amer. Math. Soc. \bf 117} (1965), 387--422. 

\bibitem{MS03} K. Murasugi and A. Stoimenow, The Alexander polynomial of planar even valence graphs, {\it Adv. Appl. Math. \bf 31} (2003), 440--462.

\bibitem{Ru18} H. Russell, private correspondence, 2018.

\bibitem{STW18} L. Traldi, D.~S. Silver and S.~G. Williams, Goeritz matrices from Dehn presentations, preprint, 2018. arXiv:1808.10296

\bibitem{Tu48} W.~T. Tutte, The dissection of equilateral triangles into equilateral triangles, {\it Proc. Cambridge Philos. Soc. \bf 44} (1948), 463--482.





\end{thebibliography}
\end{document}